\documentclass[11pt]{nyjm}
\usepackage{color}

\usepackage{graphicx}
\usepackage{caption}
\usepackage{doi}
\usepackage{comment}

\newtheorem{thm}{Theorem}[section]
\newtheorem*{thm*}{Theorem}
\newtheorem{lem}[thm]{Lemma}
\newtheorem*{lem*}{Lemma}

\newtheorem*{prop*}{Proposition}

\theoremstyle{definition}
\newtheorem{defn}[thm]{Definition}

\title{Stick number of non-paneled knotless spatial graphs}
\author{ Erica Flapan}
\address{Department of Mathematics,
610 N. College Ave.,
Pomona College,
Claremont, CA USA}
\email{elf04747@pomona.edu}
\author{Kenji Kozai}
\address{Department of Mathematics,
Rose-Hulman Institute of Technology,
5500 Wabash Ave.,
Terre Haute, IN USA}
\email{kozai@rose-hulman.edu}
\author{Ryo Nikkuni}
\address{Department of Mathematics, 
Tokyo Woman's Christian University, 
2-6-1 Zempukuji, Suginami-ku, Tokyo 167-8585, Japan}
\email{nick@lab.twcu.ac.jp}

 \subjclass{57M15, 57K10, 05C10, 92C40, 92E10}

    \keywords{Non-paneled knotless spatial graphs, stick embeddings of graphs, metalloproteins, m\"{o}bius ladders, ravels, $K_{3,3}$, $K_4$, $K_5$}
    
    \thanks{The first author was supported in part by NSF Grant DMS-1607744. The third author was supported by JSPS KAKENHI Grant Number JP15K04881.}
    
\begin{document}
\begin{abstract} We show that the minimum number of sticks required to construct a non-paneled knotless embedding of $K_4$ is 9 and of $K_5$ is 12 or 13.  We use our results about $K_4$ to show that the probability that a random linear embedding of $K_{3,3}$ in a cube is in the form of a M\"{o}bius ladder is $0.97380\pm 0.00003$, and offer this as a possible explanation for why $K_{3,3}$ subgraphs of metalloproteins occur primarily in this form.
\end{abstract}
\maketitle

\section{Introduction}

Chemists introduced the term {\it ravel} in 2008 to describe a hypothetical molecular structure whose topological complexity is the result of ``an entanglement of edges around a vertex that contains no knots or links" \cite{Ravels}.  The first such molecule was synthesized by Feng Li et al. in 2011 \cite{Li}.  In order to formalize this concept mathematically, spatial graph theorists define an embedding $G$ of an abstractly planar graph in $\mathbb{R}^3$ to be a \emph{ravel} if $G$ is non-planar but contains no non-trivial knots or links.  Such embedded graphs are closely related to {\it almost trivial graphs}, which are abstractly planar graphs with non-planar embeddings such that removing any edge from the embedded graph makes it planar.  In fact, any almost trivial graph is a ravel, though the converse is not true.  Probably the most famous example of such a graph is Kinoshita's $\theta$-curve \cite{Kinoshita, Kino}.  In order to extend the idea of a ravel to graphs which are not abstractly planar, we consider embedded graphs which contain no non-trivial knots yet contain at least one cycle which does not bound a disk in the complement of the graph.  In particular, we have the following definition.

\begin{defn}   A graph embedded in $\mathbb{R}^3$ is said to be \textit{paneled} if every cycle in the graph bounds a disk whose interior is disjoint from the graph.
\end{defn}

A number of significant results have been obtained about paneled graphs.  Of particular note, Robertson, Seymour, and Thomas proved that a graph has a linkless embedding if and only if it has a paneled embedding, which in turn occurs if and only if the graph does not contain one of the seven graphs in the Petersen family as a minor \cite{robertson95}.  In the same paper, they showed that a given embedding of a graph is paneled if and only if the complement of every subgraph has free fundamental group.  In addition, they proved that up to homeomorphism, $K_{3,3}$ and $K_5$ each have a unique paneled embedding.

A piecewise linear embedding of a knot, link, or graph in $\mathbb{R}^3$ is said to be a {\it stick} embedding.  The {\it stick number} of a knot or link is the smallest number of sticks that are required to construct it. Numerous results have been obtained about the stick number of knots and links.  But the concept of stick number can also be applied to embedded graphs.  In particular, we define the {\it non-paneled knotless  stick number} of a graph to be the minimum number of sticks required to create an embedding of the graph which is not paneled and yet contains no knots.  For example, Huh and Oh \cite{huh09} showed that the non-paneled knotless stick number of a $\theta$-graph is 8.  

We are interested in the non-paneled knotless stick number of complete graphs.  Note that $K_3$ is paneled if and only if it is knotless, and every embedding of $K_n$ with $n\geq 7$ contains a non-trivial knot \cite{conway83}.  Thus neither $K_3$ nor $K_n$ with $n\geq 7$ can have a non-paneled knotless stick number.  On the other hand, there is a linear embedding of $K_6$ which contains no knot, and we know from \cite{robertson95} that no embedding of $K_6$ is paneled.  Hence the non-paneled knotless stick number of $K_6$ is 15 (its number of edges).  So the only complete graphs whose non-paneled knotless stick number is unknown are $K_4$ and $K_5$.

In Section~2, we show that the non-paneled knotless stick number of $K_4$ is 9.  Then in Section~3, we apply this result to study embeddings of $K_{3,3}$ as subgraphs of metalloproteins, and offer a possible explanation for why such subgraphs seem to occur primarily in the form of a M\"{o}bius ladder (i.e., a M\"{o}bius strip where the surface is replaced by a ladder with three rungs).  Finally, in Section~4, we show that the non-paneled knotless stick number of $K_5$ is either 12 or 13.

\section{Stick number for non-paneled knotless embeddings of $K_4$ }
	
In order to determine the stick number of a non-paneled knotless $K_4$, we note that the minimum number of sticks of an embedding of $K_4$ is equal to the number of edges, which is 6. Thus our first step is to show that every $6$-stick and $7$-stick embedding of $K_4$ is paneled.  Then we will analyze the different cases for an $8$-stick embedding --- in the first case, the embedding of $K_4$ has a single edge consisting of three sticks; in the other cases, the embedding has two different edges each composed of two sticks, and these two edges may or may not share a vertex. Our basic strategy is to restrict the number of configurations we need to analyze by noting that if an embedding of $K_4$ can be isotoped to one that only has 7 sticks by removing one of the degree 2 vertices, then the embedding must be paneled. We use the following definition from \cite{huh09}.

\begin{defn}
	Let $P$ be a linear embedding of a graph $G$. A triangle determined by two adjacent sticks is said to be
	\textit{reducible} if its interior is disjoint from the edges of $P$. Otherwise, it is said to be \textit{irreducible}.
\end{defn}

Let $P$ be a linear embedding of a graph $G$ with a degree $2$ vertex $v$ adjacent to vertices $v_1$ and $v_2$.   If the triangle determined by the edges $\overline{vv_1}$ and $\overline{vv_2}$ is reducible, then the image of the embedding $P$ is ambient isotopic to one where the $2$-stick segment $\overline{v_1vv_2}$ is replaced by a single stick $\overline{v_1v_2}$. Hence, if a non-paneled knotless stick embedding of a graph has a minimum number of sticks, then every triangle determined by the edges incident to a degree $2$ vertex must be irreducible. We utilize this property to restrict the possible configurations of non-paneled embeddings of $K_4$.

\begin{lem}\label{nonpaneledK4}
	A non-paneled stick embedding of $K_4$ must contain at least $8$ sticks.
\end{lem}

\begin{proof}
	An embedding of $K_4$ which has only $6$ sticks is a linear embedding, which is paneled because it necessarily has the form of a tetrahedron.
	
	Suppose that there exists a non-paneled stick embedding of $K_4$ which has only $7$ sticks. Then precisely one of the edges is composed of $2$ sticks. We can think of the stick graph as a linear embedding of a graph with five vertices one of which has degree $2$. Label the vertices $v_1$, $v_2$, $v_3$, $v_4$, $v_5$, where
$v_2$ is the vertex of degree $2$ and is adjacent to $v_1$ and $v_3$. The triangle $\langle v_1,v_2,v_3\rangle$ must be irreducible, since otherwise the embedding would be isotopic to an embedding with only $6$ sticks which we saw is paneled.  Thus the edge $\overline{v_4v_5}$ must
    	pierce the triangle $\langle v_1,v_2,v_3\rangle$ as illustrated in Figure \ref{fig:k4_7}.

But now adding back the linear segments connecting vertices $v_1$ and $v_3$ to vertices $v_4$ and $v_5$ (illustrated with dotted segments in Figure \ref{fig:k4_7}) yields an embedding of $K_4$
	that is isotopic to a paneled one, contradicting the assumption that our embedding is not paneled.  \end{proof}
	
	\begin{figure}[h]
		\includegraphics[scale=.9]{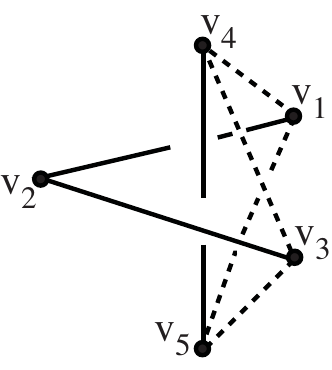}
		\caption{Even though the triangle $\langle v_1,v_2,v_3\rangle$ is pierced by $\overline{v_4v_5}$, this embedding is paneled.}
		\label{fig:k4_7}
	\end{figure}

The following lemmas deal with the three combinatorially distinct types of $8$-stick embeddings of $K_4$.

\begin{lem}\label{3Sticks}
	An $8$-stick embedding of $K_4$ with a single edge consisting of $3$ sticks is either paneled or contains a knot.
\end{lem}

\begin{proof} We consider an $8$-stick embedding of $K_4$ as a graph on 6 vertices. 
	We denote the vertices of $K_4$ by $v_1$, $v_2$, $v_3$, $v_4$, where the edge between $v_3$ and $v_4$ has $3$
	sticks with intermediate degree 2 vertices $v_5$ and $v_6$, so that the edge from $v_3$ to $v_4$ is the path $\overline{v_3v_5v_6v_4}$.
	
	We saw in Lemma~\ref{nonpaneledK4} that the stick number of a non-paneled $K_4$ is at least $8$.  Thus if the embedding of $K_4$ is non-paneled, then the
	triangles $\langle v_4,v_6,v_5\rangle$ and $\langle v_3,v_5,v_6\rangle$ must be irreducible. Without loss of generality, either $\overline{v_1v_2}$ or $\overline{v_2v_3}$
	pierces the triangle $\langle v_4,v_6,v_5\rangle$.  In either case, up to an affine transformation, we may  
	assume that the piercing edge is orthogonal to the triangle $\langle v_4,v_6,v_5\rangle$,
	which lies in the horizontal plane. 
	\medskip 
	
	\noindent \textbf{Case 1:} $\overline{v_2v_3}$ pierces $\langle v_4,v_6,v_5\rangle$ and $\overline{v_1v_2}$ does not pierce $\langle v_4,v_6,v_5\rangle$.

	In order for the triangle $\langle v_3, v_5, v_6 \rangle$ to be irreducible,
	$\overline{v_1v_4}$ must intersect the triangle. There are two subcases according to whether $\overline{v_1v_4}$ goes in front or behind $\overline{v_2v_3}$.  In the first subcase the embedding contains a knot, while in the second subcase the embedding is paneled (see Figure \ref{fig:k431}).

	\begin{figure}[h]
			\includegraphics[scale=.9]{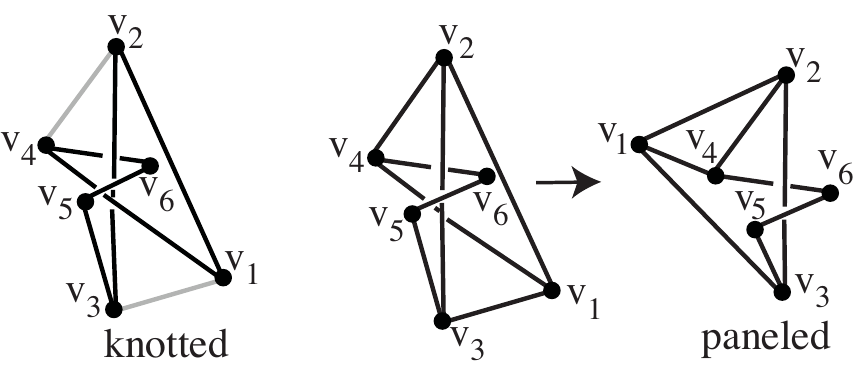}
		\caption{If $\overline{v_2v_3}$ pierces $\langle v_4, v_6, v_5 \rangle$ and $\overline{v_1v_2}$ does not, the embedding either contains a knot or is paneled.}
			\label{fig:k431}
		\end{figure}

\noindent \textbf{Case 2:}  Both $\overline{v_1v_2}$ and $\overline{v_2v_3}$ pierce $\langle v_4, v_6, v_5\rangle$

As in Case 1, there are two subcases according to whether $\overline{v_1v_4}$ goes in front or behind $\overline{v_2v_3}$.  However, both embeddings are isotopic to the same paneled embedding as illustrated in Figure~\ref{fig:k433}.
		
		\begin{figure}[h]
			\includegraphics[scale=.9]{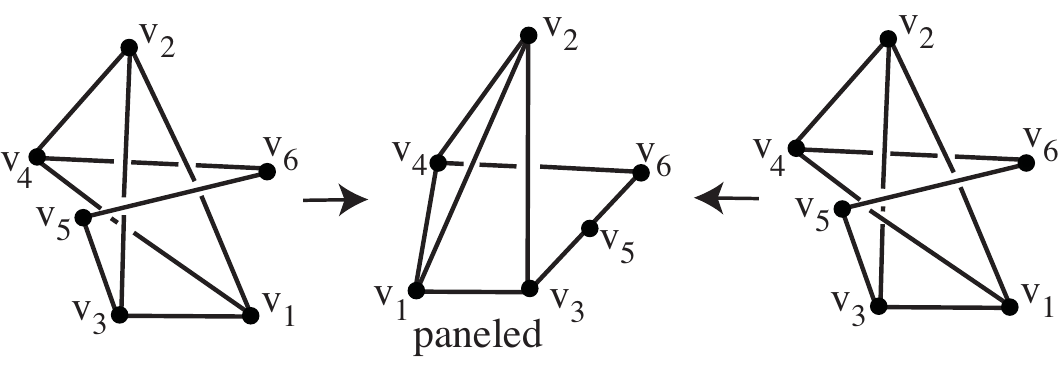}
	\caption{If both $\overline{v_1v_2}$ and $\overline{v_2v_3}$ pierce $\overline{v_4v_6v_5}$, the embedding is paneled.}
	\label{fig:k433}
	\end{figure}

	\medskip

	\noindent \textbf{Case 3:} $\overline{v_1v_2}$ pierces $\langle v_4, v_6, v_5 \rangle$ and $\overline{v_2v_3}$ does not pierce $\langle v_4,v_6, v_5\rangle$.

	In this case, the irreducibility of the triangle $\langle v_3,v_5,v_6\rangle$ implies that, up to symmetry, either $\overline{v_1v_4}$ or $\overline{v_1v_2}$ pierces it. If $\overline{v_1v_4}$ pierces $\langle v_3,v_5,v_6\rangle$, then (since $\overline{v_1v_2}$ is orthogonal to $\langle v_4,v_6,v_5\rangle$) so must $\overline{v_1v_2}$.  Hence we obtain the  paneled embedding in Figure~\ref{fig:k4_3_6}.  
	\begin{figure}[h]
			\includegraphics[scale=.9]{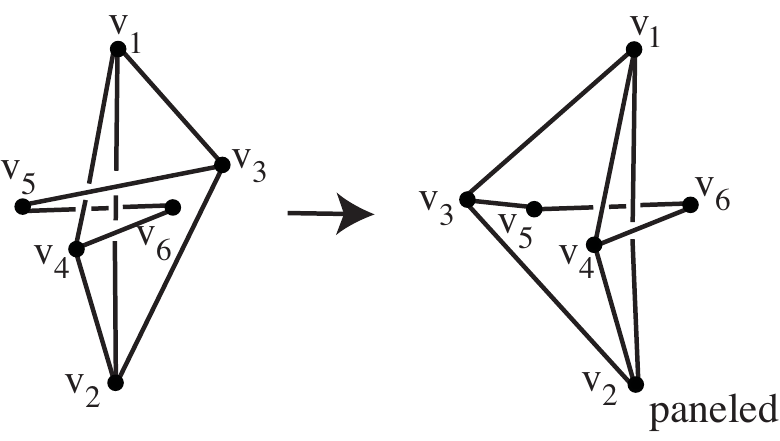}
			\caption{If $\overline{v_1v_4}$ pierces $\langle v_3,v_5,v_6\rangle$, then so must $\overline{v_1v_2}$.  Hence the embedding is paneled. }
			\label{fig:k4_3_6}
		\end{figure}
		
		Otherwise, $\overline{v_1v_2}$ pierces $\langle v_3,v_5,v_6\rangle$ and $\overline{v_1v_4}$ does not.  Now there are two subcases according to whether $v_3$ is below or above the plane determined by the triangle $\langle v_4,v_6,v_5\rangle$.  However, in both subcases the embedding contains a knot, as illustrated in Figure~\ref{fig:k4_3_5}.
			\end{proof}

		\begin{figure}[h]
			\includegraphics[scale=.9]{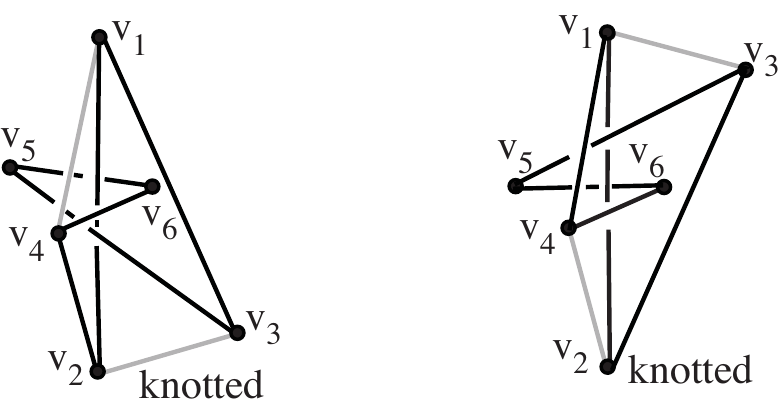}
			\caption{If $\overline{v_1v_2}$ pierces $\langle v_3,v_5,v_6\rangle$, then the embedding contains a knot.}
			\label{fig:k4_3_5}
		\end{figure}

\begin{lem}\label{K4Subgraph}
	An $8$-stick embedding of $K_4$ with two disjoint edges each consisting of two sticks is either
	paneled or contains a knot.
\end{lem}

\begin{proof}  We again consider an $8$-stick embedding of $K_4$ as a graph on 6 vertices.  Label the vertices $v_1$, $v_2$, $v_3$, $v_4$, $v_5$, $v_6$ so that $v_2$ and $v_5$ each have degree 2 and are adjacent to
	$v_1$, $v_3$ and $v_4$, $v_6$, respectively. Thus there is no edge between $v_1$ and $v_3$ and no edge between $v_4$ and $v_6$.  
	
	Suppose the embedding is non-paneled.  Then the triangle $\langle v_1,v_2, v_3\rangle$ must be
	irreducible, since otherwise the $K_4$ would be isotopic to a $7$-stick embedding, which we saw is necessarily paneled.  Thus, without loss of generality, the edge $\overline{v_4v_5}$ pierces the triangle $\langle v_1,v_2, v_3\rangle$. Up to an affine transformation, we may further 
	assume that $\overline{v_4v_5}$ is orthogonal to the triangle $\langle v_1,v_2, v_3\rangle$. 
	
	Now, the triangle $\langle v_4,v_5, v_6\rangle$ must also be irreducible, since otherwise the embedding would again be isotopic to a $7$-stick embedding.
	So either $\overline{v_1v_2}$ or $\overline{v_2v_3}$ pierces the triangle $\langle v_4,v_5, v_6\rangle$. Up to symmetry, these two cases are the same; so we assume $\overline{v_2v_3}$ pierces the triangle $\langle v_4,v_5, v_6\rangle$.

	Note we can always isotope $v_6$ out of the plane determined by $\langle v_1,v_2,v_3\rangle$.  If $v_6$ is below the plane of $\langle v_1,v_2,v_3\rangle$, then we have to consider the cases where the edge $\overline{v_1v_6}$ is in front of or behind the edge $\overline{v_4v_5}$.  However, as we see in Figure~\ref{8sticks1}, both of these cases are paneled.
	
	\begin{figure}[h]
			\includegraphics[scale=.9]{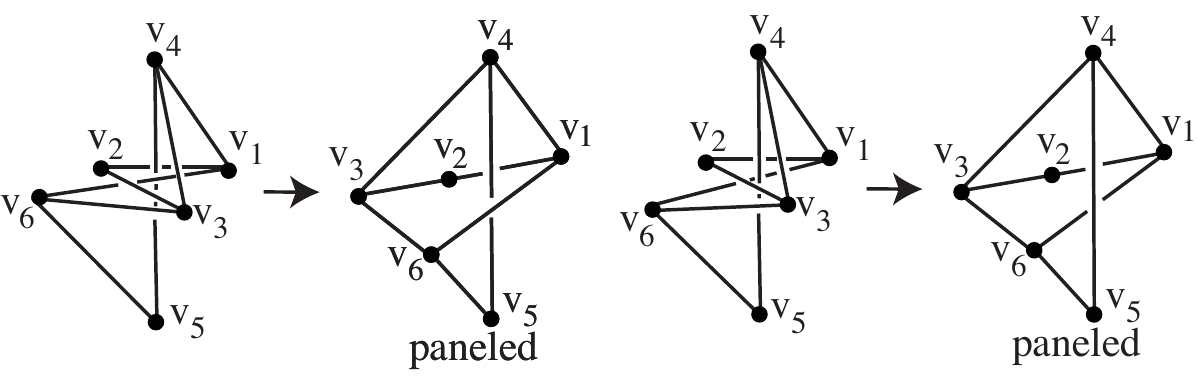}
			\caption{If $v_6$ is below the plane, then the embedding is paneled.}
			\label{8sticks1}
			\end{figure}
	
	If $v_6$ is above the plane determined by $\langle v_1,v_2,v_3\rangle$, then we consider whether the edge $\overline{v_1v_6}$ is in front of or behind $\overline{v_4v_5}$ and $\overline{v_3v_4}$.	  As shown in Figures~\ref{8sticks3}-\ref{8sticks5}, two of the embeddings are paneled and
	the third contains a knot.\end{proof}

		\begin{figure}[h]
			\includegraphics[scale=.9]{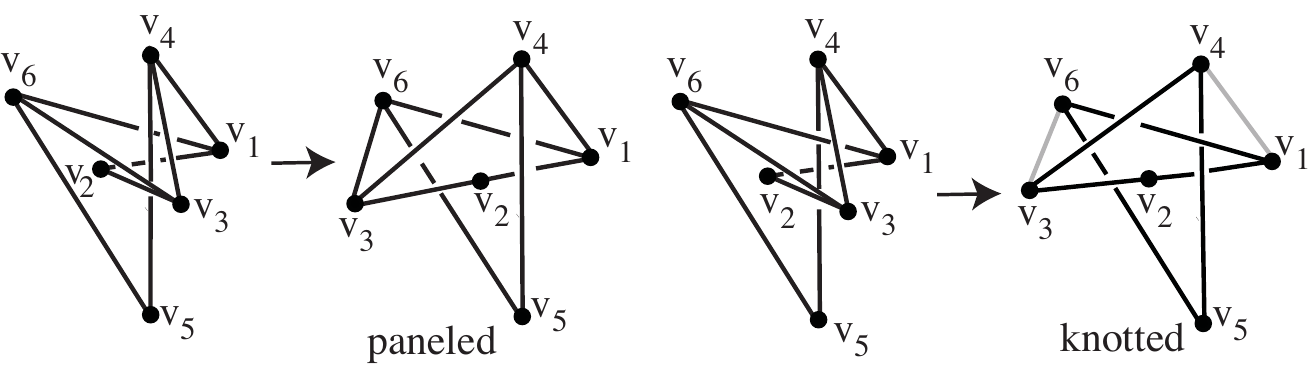}
			\caption{Here $v_6$ is above the plane.  If $\overline{v_1v_6}$ is behind $\overline{v_4v_5}$ and $\overline{v_3v_4}$, the embedding is paneled.  If $\overline{v_1v_6}$ is in front of $\overline{v_4v_5}$ but behind $\overline{v_3v_4}$, the embedding contains a knot.}\label{8sticks3}
			\end{figure}

		\begin{figure}[h]
			\includegraphics[scale=.9]{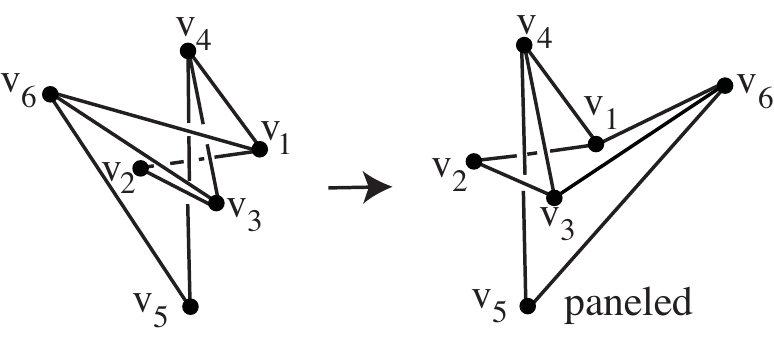}
	\caption{If $v_6$ is above the plane and $\overline{v_1v_6}$ is in front of both $\overline{v_4v_5}$ and $\overline{v_3v_4}$, then the embedding is paneled.}
	\label{8sticks5}
			\end{figure}

\begin{lem}\label{Adjacent}
	An $8$-stick embedding of $K_4$ such that two adjacent edges consist of two sticks each is either paneled or contains a knot.
\end{lem}

\begin{proof}  Again we consider an $8$-stick embedding of $K_4$ as a graph on 6 vertices.  
	We let the degree $3$ vertices of $K_4$ be labeled $v_1$, $v_2$, $v_3$, $v_4$, the edge between $v_1$ and $v_2$ contain vertex $v_5$, and the edge between $v_2$
	and $v_3$ contain vertex $v_6$.
	
	If the embedding is non-paneled, then the triangles $\langle v_1,v_5,v_2\rangle$ and $\langle v_2,v_6,v_3\rangle$ must both be irreducible. Hence one of the edges $\overline{v_3v_4}$
	or $\overline{v_3v_6}$ must pierce the triangle $\langle v_1,v_5, v_2\rangle$, and one of the edges $\overline{v_1v_4}$ or $\overline{v_1v_5}$ must pierce the triangle $\langle v_2, v_6, v_3 \rangle$.  In each case, we may assume that, up to affine transformation, the edge is orthogonal to the triangle it pierces.  Thus if only $\overline{v_3v_4}$ pierces the triangle $\langle v_1,v_5, v_2\rangle$, then $\overline{v_1v_4}$  cannot pierce $\langle v_2,v_6,v_3\rangle$; and if $\overline{v_3v_6}$ pierces the triangle $\langle v_1,v_5, v_2\rangle$, then $\overline{v_1v_5}$ cannot pierce $\langle v_2,v_6,v_3\rangle$.
		\medskip
	
	\noindent \textbf{Case 1:} Only the edge $\overline{v_3v_4}$ pierces the triangle $\langle v_1,v_5, v_2\rangle$.
	
	\begin{figure}[h]
		
			\includegraphics[scale=.85]{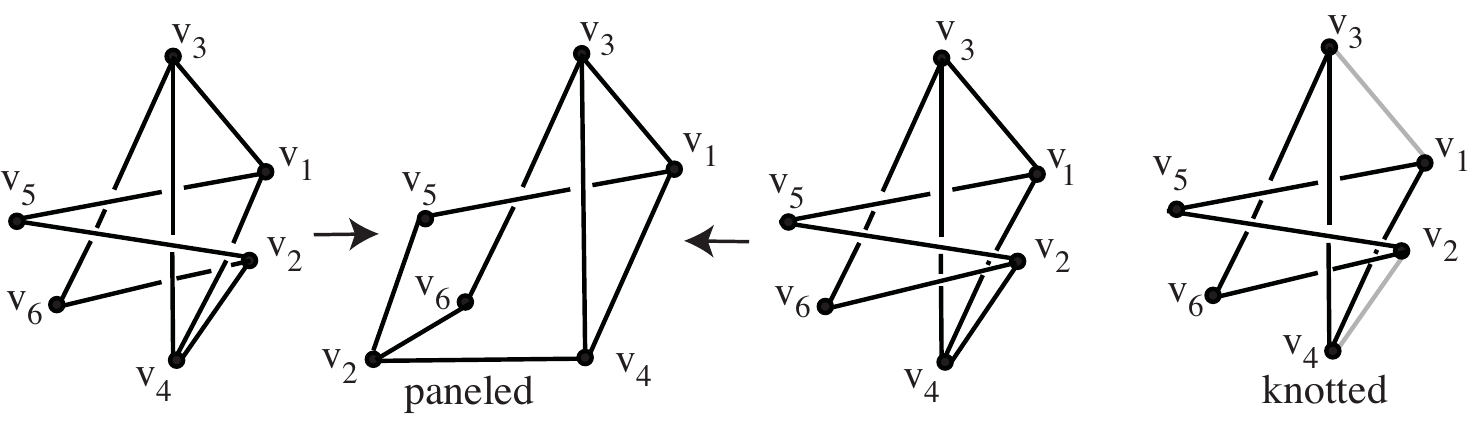}
			\caption{$\overline{v_1v_5}$ pierces $\langle v_2,v_6,v_3\rangle$, so $\overline{v_2v_6}$ must pass either behind, in front of, or between $\overline{v_1v_4}$ and $\overline{v_3v_4}$.}
			\label{fig:k4_7_2_1}
	\end{figure}

The location of $v_6$ must be such that $\langle v_2,v_6,v_3\rangle$ is pierced by $\overline{v_1v_5}$. Since $v_3$ is above the plane of $\langle v_1,v_5, v_2\rangle$, this means $v_6$ must be below the plane of $\langle v_1,v_5, v_2\rangle$.  Up to
	isotopy, this yields the configurations in Figure \ref{fig:k4_7_2_1}, where the different cases are determined by whether the edge $\overline{v_6v_2}$ is behind both $\overline{v_1v_4}$ and $\overline{v_3v_4}$ (illustrated in the first diagram), is in front of both $\overline{v_1v_4}$ and $\overline{v_3v_4}$ (illlustrated in the third diagram), or passes between $\overline{v_1v_4}$ and $\overline{v_3v_4}$ (illustrated in the fourth diagram).  The first and third configurations are isotopic to the second and paneled, and the fourth configuration contains a knot.

	\medskip

	\noindent \textbf{Case 2:} $\overline{v_3v_6}$ pierces the triangle $\langle v_1,v_5, v_2\rangle$.

In this case, the location of $v_4$ must be such that $\langle v_2,v_6,v_3\rangle$ is pierced by $\overline{v_1v_4}$.   Thus $\overline{v_2v_5}$ must pass either in front of, behind, or between $\overline{v_1v_4}$ and $\overline{v_3v_4}$.  If it is in front or behind both $\overline{v_1v_4}$ and $\overline{v_3v_4}$ (illustrated in the first and third diagrams in Figure~\ref{fig:k4_7_3_2}), then the configuration is paneled.  If it is between $\overline{v_1v_4}$ and $\overline{v_3v_4}$ (illustrated in the fourth diagram of Figure~\ref{fig:k4_7_3_2}), then the configuration contains a knot.  \end{proof}

	\begin{figure}[h]
			\includegraphics[scale=.85]{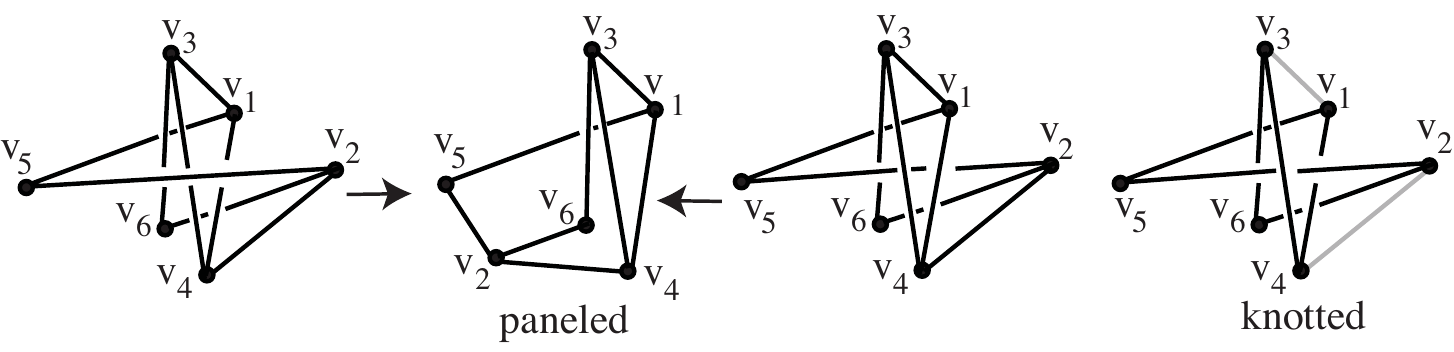}	
			\caption{$\overline{v_1v_4}$ pierces $\langle v_2,v_6,v_3\rangle$, so $\overline{v_2v_5}$ must pass either in front of, behind, or between $\overline{v_1v_4}$ and $\overline{v_3v_4}$.}	
			\label{fig:k4_7_3_2}
		\end{figure}

\begin{thm}\label{prop:k4_stick_number} There is a $9$-stick non-paneled knotless embedding of $K_4$, but no such embedding of $K_4$ exists with fewer than $9$ sticks.\end{thm}

\begin{proof} It follows from Lemmas \ref{nonpaneledK4}, \ref{K4Subgraph}, \ref{3Sticks}, and \ref{Adjacent} that no embedding of $K_4$ with 8 or fewer sticks can be both non-paneled and knotless.

Huh and Oh \cite{huh09} created an $8$-stick embedding of a $\theta$-graph which is isotopic to Kinoshita's $\theta$-graph, and hence is knotless and non-paneled.   The drawing on the left in Figure~\ref{9stick} illustrates a $9$-stick embedding of $K_4$ which is obtained from Huh and Oh's $\theta$-graph by adding an edge joining vertices $v_3$ and $v_4$ (illustrated as a dotted segment).  

\begin{figure}[h]
\includegraphics[scale=.9]{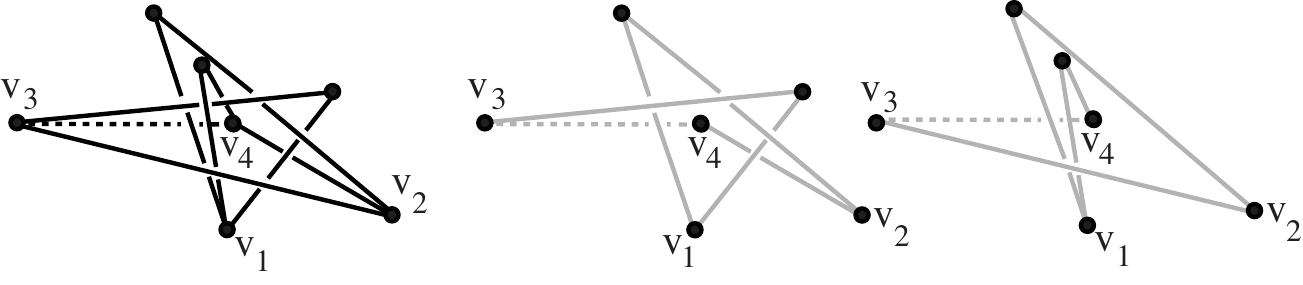}
\caption{A non-paneled knotless $9$-stick embedding of $K_4$.}
\label{9stick}
\end{figure}

Since this $K_4$ contains a non-paneled $\theta$-graph, it must be non-paneled.   On the right in Figure~\ref{9stick}, we illustrate the only cycles with at least $6$ sticks which were not in Huh and Oh's $\theta$-graph, both of which are unknotted.  No knotted cycle can have fewer than $6$ sticks.  All of the other cycles in this $K_4$ are also unknotted since they were contained in Huh and Oh's $\theta$-graph.\end{proof}

\section{Applications to metalloproteins}

Many metalloprotein structures contain one of the graphs $K_{3,3}$ or $K_5$ and hence are non-planar.  For example, nitrogenase is a complex metalloprotein containing a MoFe protein, which itself contains two M-clusters and two P-clusters.  Each M-cluster contains $K_{3,3}$, and if we include the protein backbone, each P-cluster also contains $K_{3,3}$.  Thus nitrogenase itself contains four separate $K_{3,3}$'s \cite{MBMB}.

To understand the topological complexity of such metalloproteins, it is not sufficient to know that the structure contains one or more copies of $K_5$ or $K_{3,3}$, we need to know how these subgraphs are embedded in $\mathbb{R}^3$.   However, looking at proteins as diverse as nitrogenase, cyclotides with a ``cysteine knot motif,'' and nerve growth factor we observe that  each of them contains a $K_{3,3}$ whose embedding resembles a M\"{o}bius strip as illustrated in Figure~\ref{MobiusForm}.  

\begin{figure}[h]
		\includegraphics[scale=.9]{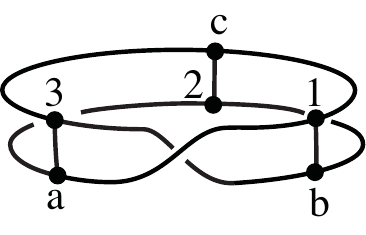}
		\caption{An embedding of $K_{3,3}$ is in {\it M\"{o}bius form} if it is isotopic to this embedding or its mirror image.}
		\label{MobiusForm}
	\end{figure}

\begin{defn} We say an embedding of $K_{3,3}$ in $\mathbb{R}^3$ is in \textit{M\"{o}bius form}, if it is isotopic to the embedding illustrated in Figure~\ref{MobiusForm} or its mirror image.\end{defn}

As an example, in Figure~\ref{Metallo} on the top left we illustrate part of a nitrogenase molecule that contains a $K_{3,3}$.  In order to better see the $K_{3,3}$ subgraph, we progressively remove more of the graph around it.  After that, we can easily isotope the $K_{3,3}$ to the form illustrated in Figure~\ref{MobiusForm}.  The question that we are interested in here is why the M\"{o}bius form of $K_{3,3}$ is so prevalent in metalloproteins that contain a $K_{3,3}$.

\begin{figure}[h]
		\includegraphics{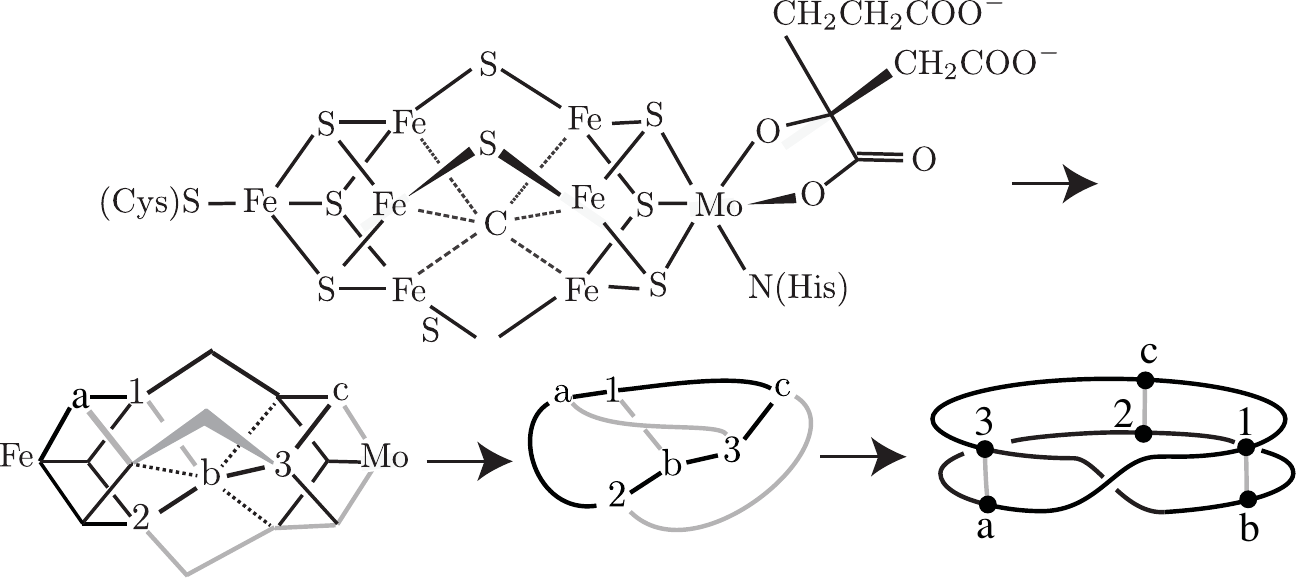}
		\caption{One of the $K_{3,3}$'s in the metalloprotein nitrogenase.}
		\label{Metallo}
	\end{figure}

 Note that while the M\"{o}bius form is knotless, not every embedding of $K_{3,3}$ which is knotless is necessarily in M\"{o}bius form.  We illustrate such an embedding in Figure~\ref{nonmobius}. To see that this embedding is knotless, observe that any knot would have to contain at least three crossings.  Hence a knot would necessarily contain two of the three edges incident to $v$.  But if we remove any one of the three edges containing $v$, the crossings between the remaining two can be removed.

\begin{figure}[h]
		\includegraphics[scale=.9]{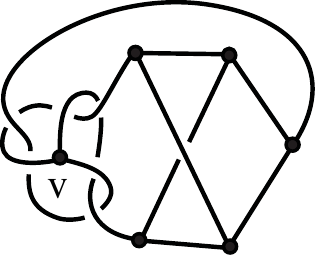}
		\caption{An embedding of $K_{3,3}$ which is not in M\"{o}bius form yet contains no knot.}
		\label{nonmobius}
	\end{figure}

\medskip

In order to shed some light on the question as to why the M\"{o}bius form of $K_{3,3}$ is prevalent in metalloproteins, we model embeddings of $K_{3,3}$ in a metalloprotein by random linear embeddings of $K_{3,3}$ in a cube and prove the following theorem.

\begin{thm}\label{thm:k33} The probability that a random linear embedding of $K_{3,3}$ in a cube is in M\"{o}bius form is $0.97380\pm 0.00003$.
\end{thm}

As the first step in proving this theorem, we will show that any knotless linear embedding of $K_{3,3}$ is in M\"{o}bius form.  Then we randomly distribute $6$ points in a cube and connect all pairs of points with straight segments to create a random linear embedding of $K_6$.  Finally, we compute the probability that such a $K_6$ contains a knot and use it to compute the probability that a random linear embedding of $K_{3,3}$ is in M\"{o}bius form.

For the first step, we will use the definition and results below.

\begin{defn} An embedded graph $\Gamma$ in $\mathbb{R}^3$ is said to be \textit{free} if the fundamental group of the complement of $\Gamma$ is free, and \textit{totally free} if every subgraph of $\Gamma$ is free.
\end{defn}

\begin{lem*} [Robertson, Seymour, and Thomas  \cite{robertson95}]  An embedded graph is paneled if and only if it is totally free.  
\end{lem*}   
	
	\begin{thm*} [Robertson, Seymour, and Thomas  \cite{robertson95}] Up to homeomorphism, $K_{3,3}$ and $K_5$ each have a unique paneled embedding.
\end{thm*}

\begin{thm*}[Huh and Lee \cite{huh15}] Every linear embedding of a graph with no more than $6$ vertices all of degree at least $3$ is free.
\end{thm*}

We now prove the following.

\begin{thm} \label{knotless} Every knotless linear embedding of $K_{3,3}$ is isotopic to an embedding in M\"{o}bius form.\end{thm}

\begin{proof} Fix a knotless linear embedding $\Gamma$ of $K_{3,3}$ in $\mathbb{R}^3$.  Observe that if we delete a single edge of $\Gamma$ we obtain a knotless $8$-stick embedding of $K_4$ with two disjoint edges each consisting of two sticks (see Figure~\ref{subgraphK4}).  It now follows from Lemma~\ref{K4Subgraph} that such a $K_4$ subgraph must be paneled.

\begin{figure}[h]
		\includegraphics[scale=.9]{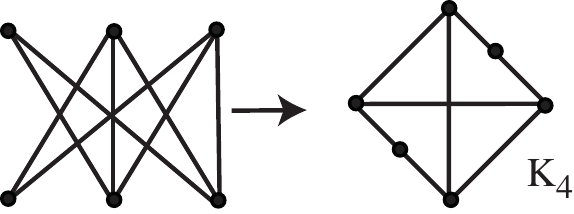}
		\caption{If we delete a single edge of $K_{3,3}$ we obtain a $K_4$.}
		\label{subgraphK4}
	\end{figure}
	
	Thus by the above lemma of Robertson, Seymour, and Thomas \cite{robertson95}, every $K_4$ subgraph of $\Gamma$ is totally free.  Since every proper subgraph of $\Gamma$ is obtained by removing at least one edge, such a subgraph  must be contained in a $K_4$ subgraph, and hence must be free.  Also, we know from Huh and Lee's Theorem \cite{huh15} that $\Gamma$ itself is free.  Thus $\Gamma$ is totally free.  
	
	By using Robertson, Seymour, and Thomas's Lemma \cite{robertson95} again we see that $\Gamma$ is paneled.  Finally, observe that the form of $K_{3,3}$ in Figure~\ref{MobiusForm} is paneled, and hence by Robertson, Seymour, and Thomas's Uniqueness Theorem \cite{robertson95}, $\Gamma$ must be in M\"{o}bius form.\end{proof}

We are now ready to prove Theorem \ref{thm:k33}, showing that the M\"{o}bius form of $K_{3,3}$ is overwhelmingly the most common  among random linear embeddings.

\begin{proof}[Proof of Theorem \ref{thm:k33}] Hughes \cite{Hughes}, Huh and Jeon \cite{huh07}, and Nikkuni \cite{Nikkuni} independently proved
	that a linear embedding of $K_6$ contains at most one knot (which is a trefoil knotted $6$-cycle), and the embedding contains such a knot if and only if it contains
	three Hopf links.  It was shown in \cite{flapan16} that the probability that a random linear embedding
	of $K_6$ in a cube has exactly three Hopf links is $\frac{45q-1}{2}$, where
	$q = 0.033867 \pm 0.000013$.
	
Now we randomly distribute $6$ vertices in a cube and add linear segments between each pair of vertices to obtain a random linear embedding of $K_6$.   We then remove a pair of disjoint $3$-cycles to get a random linear embedding of $K_{3,3}$. The only way this embedding could contain a knot is if the embedding of $K_6$ contained a knot and the pair of disjoint $3$-cycles that were removed left the knot intact.  There are $10$ distinct pairs of disjoint $3$-cycles in $K_6$, only one of which shares no edges with a given $6$-cycle.    It follows that the probability that $K_{3,3}$ is knotted is $\frac{45q-1}{20}=0.02620 \pm 0.00003$.  By using Theorem~\ref{knotless} the result follows.
\end{proof}

\section{Stick number for non-paneled knotless embeddings of $K_5$}

We have seen that the non-paneled knotless stick
number of $K_{4}$ is 9.  Thus, the remaining case of non-paneled knotless stick embeddings of $K_n$, is when $n=5$.  In this section, we show that the number of sticks in such an embedding of $K_5$ is at least 12 and at most 13.

\begin{thm}\label{K5_paneled}
There are precisely two linear embeddings of $K_5$ in general position up to affine transformation.  Futhermore, all linear embeddings of $K_5$ are paneled and ambient isotopic.
\end{thm}

\begin{proof} We begin with vertices $v_1$, $v_2$, $v_3$, $v_4$ in general position in $\mathbb{R}^3$.  The solid tetrahedron these vertices determine can be thought of as the intersection of four half-spaces, each defined by the plane containing three of the four vertices. For each such plane, we call the half-space that contains the tetrahedron the {\it inside} half-space, and the half-space that does not contain the tetrahedron the {\it outside} half-space.

Vertex $v_5$ cannot be in the plane of any three of the existing vertices, since the vertices are in general position. If vertex $v_5$ is inside the tetrahedron $\langle v_1,v_2,v_3,v_4\rangle$, then it is in the inside half-space of the planes determined by all of the faces (see Figure~\ref{K5}a). If vertex $v_5$ is outside the half-space of the face $\langle v_2,v_3,v_4\rangle$ and inside the half-space of the three other faces, then the five vertices are in the configuration of two tetrahedra glued along the face $\langle v_2,v_3,v_4\rangle$ together with the edge $\overline{v_1v_5}$ (see Figure~\ref{K5}b). If vertex $v_5$ is outside the half-spaces of the faces $\langle v_1,v_2,v_3\rangle$ and $\langle v_1,v_2,v_4\rangle$ and inside the half-spaces of the faces $\langle v_1,v_3,v_4\rangle$ and $\langle v_2,v_3,v_4\rangle$, then the five vertices are in the configuration of two tetrahedra glued along the face $\langle v_3,v_4,v_5\rangle$ together with the edge $\overline{v_1v_2}$ (see Figure~\ref{K5}c, where $v_5$ is above $\langle v_1,v_2,v_4\rangle$ and in front of $\langle v_1,v_2,v_3\rangle$). If vertex $v_5$ is outside the half-spaces of the faces $\langle v_1,v_2,v_3\rangle$, $\langle v_1,v_2,v_4\rangle$, and $\langle v_1,v_3,v_4\rangle$ and inside the half-space of the face $\langle v_2,v_3,v_4\rangle$, then vertex $v_5$ is the point of a cone about vertex $v_1$. Thus vertex $v_1$ lies inside tetrahedron $\langle v_2,v_3,v_4,v_5\rangle$ (see Figure~\ref{K5}d).   The case where vertex $v_5$ lies on the outside half-space of all four faces is impossible. 

 \begin{figure}[h]
\includegraphics[scale=.9]{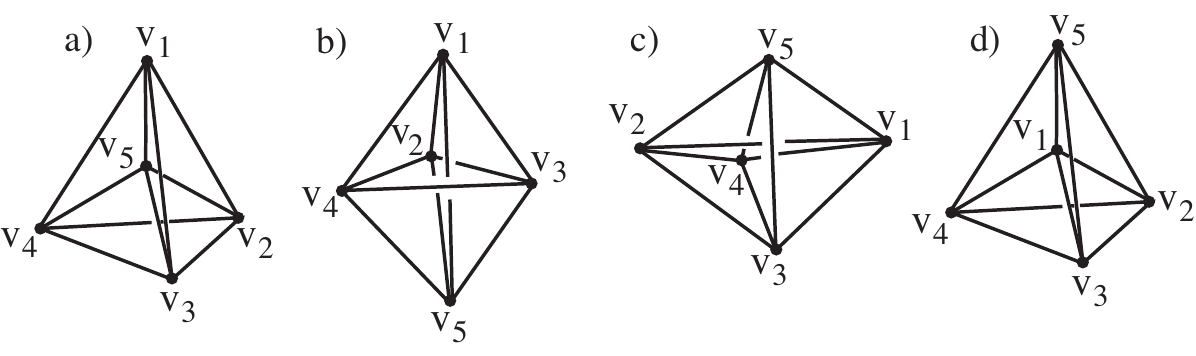}
\caption{These are all the linear embeddings of $K_5$ up to affine transformation.}
\label{K5}
\end{figure}

Observe that Figures~\ref{K5}(a) and (d) are equivalent by affine transformation, and Figures~\ref{K5}(b) and (c) are equivalent by affine transformation.  Also, all of the embeddings in Figure~\ref{K5} are paneled and ambient isotopic.
\end{proof}


 Note that a linear embedding of $K_5$ has $10$ sticks.  We show below that there are no knotless non-paneled $11$-stick embeddings of $K_5$.  
 
\begin{thm}\label{11stickK5}  Every $11$-stick embedding of $K_5$ is either paneled or contains a knot.
\end{thm}

\begin{proof}Suppose that there exists a non-paneled knotless embedding of $K_5$ with $11$ sticks. Then the stick graph is a linear embedding of a graph with six vertices, one of which has degree $2$. Label the vertices $v,v_1,v_2,v_3,v_4$ and $v_5$, where $v$ is the vertex of degree $2$ and is adjacent to $v_1$ and $v_2$. The triangle $\langle v_1,v,v_2\rangle$ must be irreducible, since otherwise the embedding would be isotopic to an embedding with only $10$ sticks which is paneled by Theorem \ref{K5_paneled}. Thus one of the edges $\overline{v_3v_4}$, $\overline{v_4v_5}$, or $\overline{v_3v_5}$ must pass through the triangle $\langle v_1,v,v_2\rangle$.  Up to re-labeling and affine transformation, we may assume that the edge $\overline{v_3v_4}$ intersects the triangle $\langle v_{1},v,v_{2}\rangle$ orthogonally.  Up to symmetry, we may also assume that $v_5$ lies above the plane determined by the triangle $\langle v_1, v, v_2\rangle$ and on the same side of the plane determined by $\langle v,v_3, v_4\rangle$ as $v_1$ is.

We now determine the possible positions of $v_5$ by orthogonal projection onto the plane determined by $\langle v_1,v,v_2\rangle$.   The projection of the vertex $v_5$ can either lie inside the infinite wedge bounded by the rays $\overrightarrow{v_3v_1}$ and $\overrightarrow{v_3v_2}$ in the projection, or outside of this wedge.  Note that in the projection $v_3$ and $v_4$ are the same point. 

\begin{figure}[h]
\includegraphics[scale=.9]{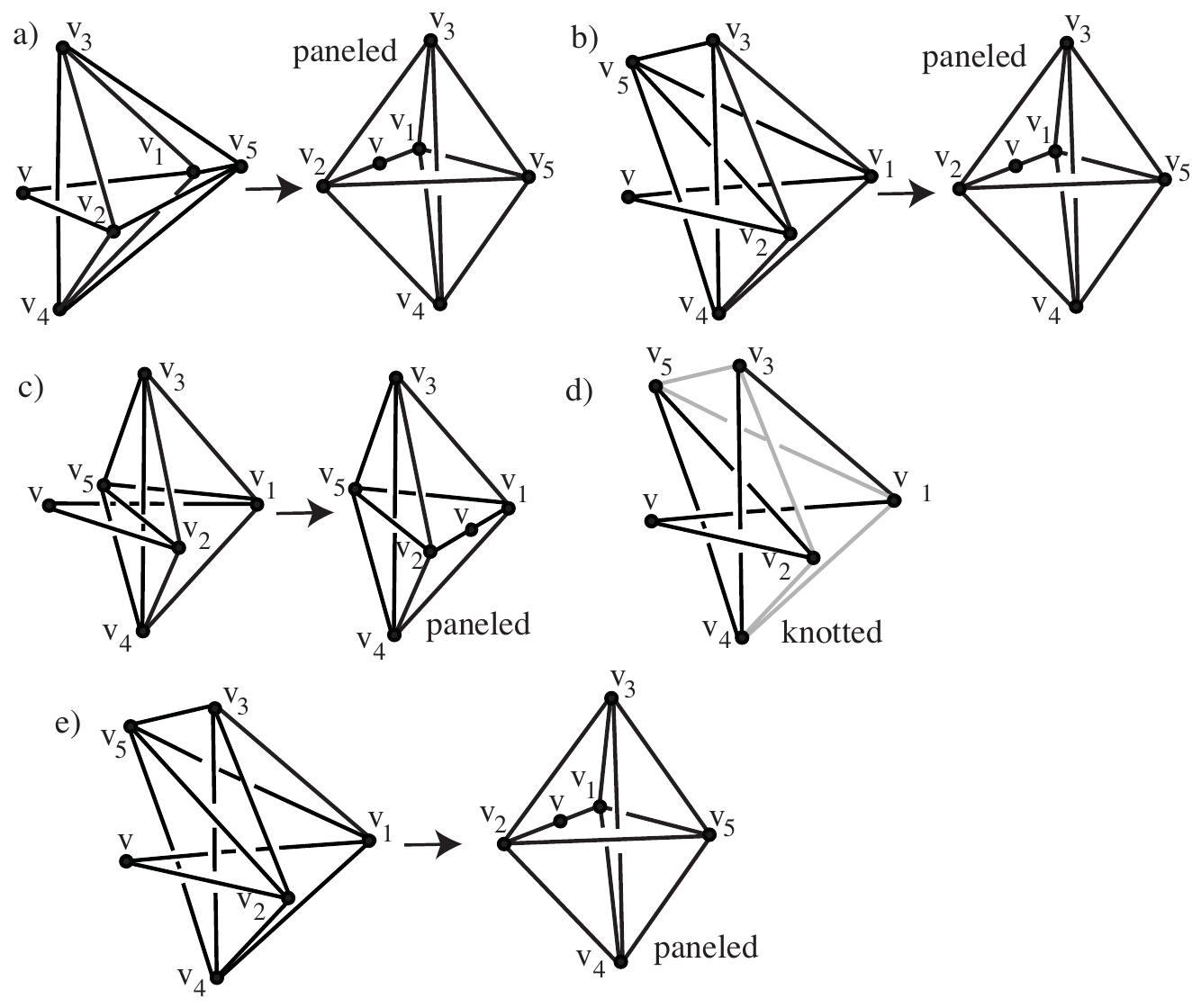}
\caption{Since the triangle $\overline{vv_{1}v_{2}}$ and the edge $\overline{v_3v_4}$ are orthogonal, there are only five positions for $v_5$ up to symmetry and affine transformation.}
\label{TriangleK5}
\end{figure}

  In the case that $v_5$ is inside the wedge, we obtain an embedding isotopic to Figure \ref{TriangleK5}(a), which is paneled.
For the case when $v_5$ is outside of the wedge, first recall that we assumed $v_5$ was on the same side of the plane determined by $\langle v,v_3, v_4\rangle$ as $v_1$.  The spatial configurations are now determined by whether $\overline{v_2v_5}$ pierces the triangle $\langle v_1,v_3,v_4\rangle$ and whether $\overline{v_4v_5}$ pierces the triangle $\langle v_1,v,v_2\rangle$. 

If $\overline{v_2v_5}$ pierces $\langle v_1,v_3,v_4\rangle$ and $\overline{v_4v_5}$ pierces $\langle v_1,v,v_2\rangle$, we obtain an embedding isotopic to Figure \ref{TriangleK5}(b), which is paneled. If $\overline{v_2v_5}$ does not pierce $\langle v_1,v_3,v_4\rangle$ and $\overline{v_4v_5}$ does pierce $\langle v,v_1,v_2\rangle$, the embedding is isotopic to Figure \ref{TriangleK5}(c), which is paneled.  If $\overline{v_2v_5}$ pierces $\langle v_1,v_3,v_4\rangle$ but $\overline{v_4v_5}$ does not pierce $\langle v,v_1,v_2\rangle$, we obtain an embedding isotopic to Figure \ref{TriangleK5}(d), which contains a trefoil knot. Finally, if neither $\overline{v_2v_5}$ pierces $\langle v_1,v_3,v_4\rangle$ nor $\overline{v_4v_5}$ pierces $\langle v,v_1,v_2\rangle$, then the embedding is isotopic to Figure \ref{TriangleK5}(e), which is paneled.\end{proof}

In Figure \ref{fig:13stick}, we illustrate a non-paneled knotless embedding of $K_5$ constructed with $13$ sticks. Observe that if we remove $v_5$ and its incident edges as well as the edge $\overline{v_1v_2}$, we obtain the $8$-stick Kinoshita's $\theta$-curve, illustrated on the right.  Since a Kinoshita's $\theta$-curve is non-paneled, this embedding of $K_5$ cannot be paneled.  

\begin{figure}[h]
\includegraphics[scale=.8]{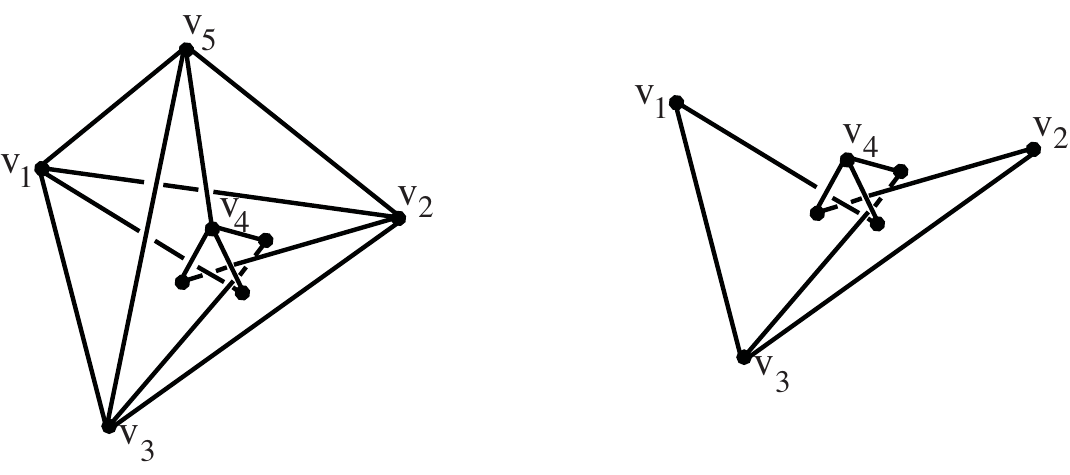}
\caption{A $13$-stick embedding of $K_5$ which contains the $8$-stick Kinoshita's $\theta$-curve on the right.}
\label{fig:13stick}
\end{figure}

In order to determine if our embedding contains a knot, we only need to consider loops with at least $6$ sticks.  Such a loop must contain vertex $v_4$.  Hence, up to symmetry and affine transformation, it has the form of one of the grey loops in Figure~\ref{fig:13stickLoops}, depending on whether or not it contains the edge $\overline{v_4v_5}$.  Since both of these loops are unknotted, this embedding is knotless.  This example together with Theorem~\ref{11stickK5} shows that the stick number of a non-paneled, knotless embedding of $K_5$ is either $12$ or $13$.

\begin{figure}[h]
\includegraphics[scale=.8]{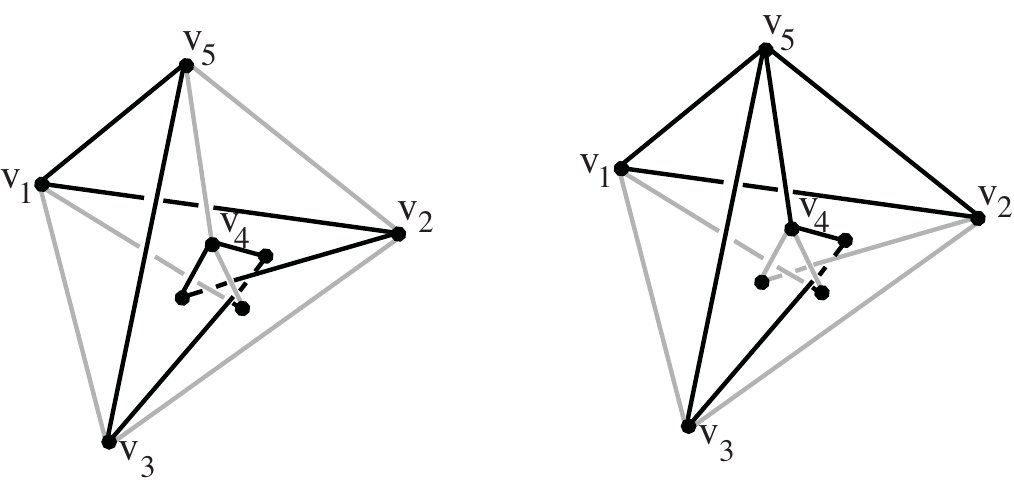}
\caption{Up to symmetry and affine transformation, the grey loops are the only loops in our $K_5$ with at least 6 sticks.}
\label{fig:13stickLoops}
\end{figure}

It remains an open question whether it is possible to have a paneled knotless embedding of $K_5$ with $12$ sticks. The analysis of embeddings of $K_5$ constructed with $12$ sticks has considerably more cases than the analysis of $K_4$ embeddings with $8$ sticks. For example, in each of the cases of Lemma \ref{3Sticks}, for $K_5$ we would have to consider all possible positions of an additional vertex $v_7$ connected to each of $v_1,v_2,v_3$, and $v_4$. Moreover, we must consider the cases where the edge piercing $\langle v_3, v_5, v_6 \rangle$ or $\langle v_4, v_6, v_5\rangle$ contains $v_7$. Each case in each of Lemmas \ref{3Sticks}--\ref{Adjacent} will have a corresponding increase in complexity. 

Alternatively, a matroid approach \cite{huh11,ramirez99} can be used to determine all $12$-stick embeddings of $K_5$ with the known census of affine point configurations \cite{finchi}.  Eliminating all reducible and knotted embeddings still results in a few thousand cases, only a small number of which can be easily isotoped to a reducible embedding. In order to extend this approach to systematically identify the isotopy types of embeddings of $K_5$ consisting of $12$ sticks, one would need a more general method for determining isotopies of an embedding of a non-complete graph that could change the combinatorial type of the affine point configuration.


\begin{thebibliography}{99}
 
\bibitem{Ravels}
{\scshape Castle, T.; Evans, M. E.; Hyde, S. T.}
Ravels: knot free but not free. Novel entanglements of graphs in 3-space.
{\em New J. Chem.} {\bf 32} (2008), 1457--1644.
\doi{10.1039/B719665B}.

\bibitem{conway83}
{\scshape Conway, J. H.; Gordon, C. McA.}
Knots and links in spatial graphs.
{\em J. Graph Theory.} {\bf 7} (1983), 445--453.
\mrev{0722061} (85d:57002),
\zbl{0524.05028},
\doi{10.1002/jgt.3190070410}.

\bibitem{finchi}
{\scshape Finchi, L.}
Homepage of oriented matroids. http://www.om.math.ethz.ch.

\bibitem{MBMB}
{\scshape Flapan, E.; Heller, G.}
Topological complexity in protein structures.
{\em Mol. Based Math. Biol.} {\bf 3} (2015), 23--43.
\mrev{3349352},
\zbl{1347.92055},
\doi{10.1515/mlbmb-2015-0002}.

\bibitem{flapan16}
{\scshape Flapan, E.; Kozai, K.}
Linking number and writhe in random linear embeddings of graphs.
{\em J. Math. Chem.} {\bf 54} (2016), 1117--1133.
\mrev{3484479},
\zbl{1346.92082},
\arx{1508.01183},
\doi{10.1007/s10910-016-0610-2}.

\bibitem{Hughes}
{\scshape Hughes, C.}
Linked triangle pairs in straight edge embeddings of {$K_6$}.
{\em Pi Mu Epsilon Journal.} {\bf 12} (2006), 213--218.

\bibitem{huh07}
{\scshape Huh, Y.; Jeon, C. B.}
Knots and links in linear embeddings of {$K_6$}.
{\em J. Korean Math. Soc.} {\bf 44} (2007), 661--671.
\mrev{2314834} (2008a:57005),
\zbl{1136.57001},
\doi{10.4134/JKMS.2007.44.3.661}.

\bibitem{huh11}
{\scshape Huh, Y.}
Heptagonal knots and Radon partitions.
{\em J. Korean Math. Soc.} {\bf 48} (2011), 367--382.
\mrev{2789461} (2012c:57005),
\zbl{1219.57009},
\arx{1007.0841},
\doi{10.4134/JKMS.2011.48.2.367}.

\bibitem{huh15}
{\scshape Huh, Y.;Lee, J. H.}
Linearly embedded graphs in 3-space with homotopically free exteriors.
{\em Algebr. Geom. Topol.} {\bf 15} (2015), 1161--1173.
\mrev{3342688},
\zbl{1315.57010},
\arx{1409.6796},
\doi{10.2140/agt.2015.15.1161}.

\bibitem{huh09}
{\scshape Huh, Y.; Oh, S.}
Stick number of theta-curves.
{\em Honam Math. J.} {\bf 31} (2009), 1--9.
\mrev{2504508} (2010b:57003),
\zbl{1179.57011},
\doi{10.5831/HMJ.2009.31.1.001}.

\bibitem{Kinoshita}
{\scshape Kinoshita, S.}
Alexander polynomials as isotopy invariants. {I}.
{\em Osaka Math. J.} {\bf 10} (1958), 263--271.
\mrev{102819} (21 \#1605),
\zbl{0119.38801}.

\bibitem{Kino}
{\scshape Kinoshita, S.}
On elementary ideals of polyhedra in the {$3$}-sphere.
{\em Pacific J. Math.} {\bf 42} (1972), 89--98.
\mrev{312485} (47 \#1042),
\zbl{0239.55002}.

\bibitem{Li}
{\scshape Li, F.;  Clegg, J.;  Lindoy, L.;  Macquart, R.; Meehan, G.}
Metallosupramolecular self-assembly of a universal 3-ravel.
{\em Nat. Commun.} {\bf 2} (2011), 1--5.
\doi{10.1038/ncomms1208}.

\bibitem{Nikkuni}
{\scshape Nikkuni, R.}
A refinement of the {C}onway-{G}ordon theorems.
{\em Topology Appl.} {\bf 156} (2009), 2782--2794.
\mrev{2556036} (2010i:57007),
\zbl{1185.57003},
\arx{0907.0152},
\doi{10.1016/j.topol.2009.08.013}.

\bibitem{ramirez99}
{\scshape Ramirez Alfonsin, J.L.}
Spatial graphs and oriented matroids: the trefoil.
{\em Discrete Comput. Geom.} {\bf 22} (1999), 149--158.
\mrev{MR1692678} (2000c:05056),
\zbl{0931.05061},
\doi{10.1007/PL00009446}.

\bibitem{robertson95}
{\scshape Robertson, N.; Seymour, P.; Thomas, R.}
Sachs' linkless embedding conjecture.
{\em J. Combin. Theory Ser. B.} {\bf 64} (1995), 185--227.
\mrev{1339849} (96m:05072),
\zbl{0832.05032},
\doi{10.1006/jctb.1995.1032}.
 
\end{thebibliography}
\end{document}